\documentclass[final,3p,times,number]{elsarticle}

\usepackage[utf8]{inputenc}
\usepackage[T1]{fontenc}
\usepackage{amssymb}
\usepackage{amsthm}
\usepackage{latexsym}
\usepackage{amsmath}
\usepackage{color}
\usepackage{graphicx}
\usepackage{indentfirst}
\usepackage{mathrsfs}
\usepackage{pdfsync}
\usepackage{bm}
\usepackage{subcaption}
\usepackage{float}     
\usepackage{placeins} 
\usepackage[titletoc]{appendix}
\numberwithin{equation}{section}
\allowdisplaybreaks

\newtheorem{theorem}{\color{black}\indent Theorem}[section]

\newtheorem{assumption}{\color{black}\indent Assumption}[section]
\newtheorem{lemma}{\color{black}\indent Lemma}[section]

\newtheorem{definition}{\color{black}\indent Definition}[section]
\newtheorem{remark}{\color{black}\indent Remark}[section]

\newtheorem*{lemma*}{Lemma}

\newcommand{\TT}{\mathbb{T}}
\newcommand{\ZZ}{\mathbb{Z}}

\newcommand{\XX}{\mathrm{X}}
\newcommand{\dd}{\,\mathrm{d}}
\newcommand{\ii}{\,\mathrm{i}}
\newcommand{\avg}[1]{\left\langle #1 \right\rangle}

\journal{*******}
\biboptions{numbers,sort&compress}
\usepackage[colorlinks=true,linkcolor=blue,citecolor=blue,urlcolor=blue]{hyperref}
\begin{document}
\begin{frontmatter}
    \title{Angular Linearization and Quantitative Convergence of Statistical Ensembles for Weighted Integrable Hamiltonian Systems}
    \author[author1]{Xinyu Liu \footnote{ E-mail address : liuxy595@jlu.edu.cn}}
    \author[author1,author3]{Yong Li\corref{cor1} \footnote{ E-mail address : liyong@jlu.edu.cn} }
    \address[author1]{College of Mathematics, Jilin University, Changchun 130012, P. R. China.}
    \address[author3]{Center for Mathematics and Interdisciplinary Sciences, Northeast Normal University, Changchun 130024, P. R. China.}
    \cortext[cor1]{Corresponding author at : Center for Mathematics and Interdisciplinary Sciences, Northeast Normal University, Changchun 130024, P. R. China.}
    \begin{abstract}
     Within the framework of weighted integrable Hamiltonian systems, we study the long-time behavior of the associated statistical ensembles. We construct an action-dependent angular conjugacy that rectifies the nonuniform angular flow into a constant-speed linear flow, thereby reducing the ensemble dynamics to oscillatory integrals over the action domain. Using integration by parts along a tailored vector field and \(W^{1,1}\) approximations with controlled boundary terms, we derive inverse-in-time decay for each Fourier mode and complete the summation under nonresonance and uniform nondegeneracy assumptions. Consequently, under suitable regularity and admissible initial data, we prove that the ensemble converges to the weighted equilibrium measure at rate \(O(1/t)\), with an explicit parameter-dependent bound. The methodology provides a transferable analytic scheme for establishing limit theorems in near-integrable regimes.
    \end{abstract}

    \end{frontmatter}


    \section{Introduction}\label{sec:1}
    \setcounter{equation}{0}
    \setcounter{definition}{0}
    \setcounter{proposition}{0}
    \setcounter{lemma}{0}
    \setcounter{remark}{0}

    In mathematics, statistical ensembles can be viewed as the study of weak convergence for the time evolution of an initial phase-space distribution under a dynamical flow: given a measurable initial density $f_0$ and a flow $\Phi_t$, the ensemble at time $t$ is $P_t := (\Phi_t)_{\#}(f_0\,\mathrm d\mu)$, and its response to an observable $G$ is $\langle G\rangle_t := \int G\,\mathrm dP_t$. The analysis of ensembles therefore focuses on the behavior of $\langle G\rangle_t$ as $t\to\infty$. This framework combines time evolution with measure theory, so that many statistical questions can be addressed in a unified way at the level of expectations, variances, and limit laws (e.g., the Law of Large Numbers and the Central Limit Theorem, as well as weak convergence)\cite{hafouta2023explicit,mitchell2019weak,LLNLIUXINYU}. From an applied point of view, ensemble methods provide prediction tools that are robust to microscopic uncertainty in complex Hamiltonian systems, plasmas, and molecular dynamics, with significant implications for control, materials, and nonequilibrium statistical physics \cite{dif2022transport,bhati2024equilibrium,defenu2403ensemble,nandy2024reconstructing} 

In recent years, LLN/CLT-type results for measurable dynamical systems have been established for many mixing or partially hyperbolic classes, following several complementary approaches, including spectral methods, transfer operators, and coupling techniques \cite{fleming2022functional,wang2025central,dolgopyat2022asymptotic}. For integrable Hamiltonian systems, existing results typically rely on Fourier analysis and small-divisor estimates; within the framework of conjugacy linearization and oscillatory integrals, they address long-time limits and convergence rates of ensembles \cite{klein2021uniform,de2022convergence,burby2023nearly}. However, in the absence of strong mixing or exponential decay, methods based on mixing are not directly applicable.

This paper focuses on a class of weighted integrable Hamiltonian systems. In action-angle coordinates $(I,\theta)\in\Omega\times\TT^N$, we introduce the weighted symplectic form
\begin{equation*}
\sigma_m \;=\; m(\theta)^{1/N}\sum_{j=1}^{N}\dd\theta_j\wedge\dd I_j,
\end{equation*}
which induces explicitly nonuniform angular velocities. Recent progress on non-volume-preserving angular dynamics and weighted symplectic structures includes, for example, mixing and spectral properties of smooth time changes \cite{avila2021mixing,fayad2021lebesgue}, asymptotics and breakdown of invariant tori via Birkhoff averages \cite{meiss2021birkhoff}, the construction of KAM quasi-periodic tori in dissipative spin-orbit models \cite{calleja2022kam}, and Hamiltonization schemes through singular cotangent models in fluid contexts \cite{coquinot2023singular}. These developments highlight the natural emergence and analytic value of nonuniform angular speeds and weighted symplectic geometry across different settings. Our model presents two core challenges: (i) Lebesgue volume is generally not preserved, so one must identify and use a weighted invariant measure; (ii) linearization requires solving the cohomological equation
\begin{equation*}
\omega(I)\cdot\nabla_\theta v_I(\theta)=b(\theta),\qquad \langle v_I\rangle=0,
\end{equation*}
and controlling small divisors under a Diophantine-type nonresonance condition, in order to construct the action-dependent angular conjugacy
\begin{equation*}
\phi=\Psi_I(\theta)=\theta+\omega(I)\,v_I(\theta).
\end{equation*}

Our main idea is as follows. We first use the conjugacy to straighten the angular dynamics into a constant-speed flow. We then express the ensemble response as oscillatory integrals over the action domain and, by integrating by parts along a suitable vector field—combined with $W^{1,1}$ approximations on Lipschitz domains and controlled boundary terms—obtain a $t^{-1}$ decay for each Fourier mode. Finally, we sum over modes at the summability threshold $r>N+1$ to derive quantitative convergence. The toolbox includes Fourier analysis and small-divisor estimates (for solving the cohomological equation), functional analysis and Sobolev-Hölder embeddings (for regularity and uniform bounds), pushforward of measures and Jacobian formulas (to construct and exploit the weighted invariant measure), and vector-field integration by parts for nonstationary oscillatory integrals (together with the density of $W^{1,1}$ by smooth compactly supported functions on Lipschitz sets).

Under appropriate Diophantine and nondegeneracy assumptions, we prove that for any $G\in C^r(\bar\Omega\times\TT^N)$ with $r>N+1$ and any initial datum $f_0\in W^{1,1}_I(\Omega\times\TT^N)$, the statistical ensemble converges quantitatively to the equilibrium associated with the weighted invariant measure at rate $O(1/t)$. The constants depend only on bounded-regularity data of the frequency map and the weight. This yields a uniform convergence rate in an integrable setting with nonuniform angular speeds—without invoking mixing—and provides analytic groundwork for extensions to near-integrable Hamiltonian systems.

The structure of the following text is as follows. Section~2 introduces the model, notation, and nonresonance assumptions. Section~3 constructs the angular conjugacy and establishes linearization and pushforward properties. Section~4 states and proves the main results. Section~5 concludes with a summary of the method and directions for future work.
    \section{Preliminaries}
    This section is devoted to a precise formulation of the system under consideration, together with the fundamental notations and standing assumptions that will be used in the following contents.
    
    Let's start with the definition of weighted Hamiltonian systems.

\begin{definition}\label{WHS}
Let $(M,\sigma)$ be a $2N$-dimensional symplectic manifold, and let $H=H(I)$ be an integrable Hamiltonian depending only on the action variables $I\in\Omega\subset\mathbb{R}^N$ in which $\Omega \subset \mathbb{R}^{N}$ is a bounded Lipschitz domain.  
Let $m:\mathbb{T}^N\to\mathbb{R}^{+}$ be a positive function, referred to as a weight function.  
We define the weighted symplectic form
\begin{equation*}
\sigma_m = m(\theta)^{1/N}\,\sum_{j=1}^N d\theta_j\wedge dI_j,\qquad (I,\theta)\in \Omega\times\TT^N.
\end{equation*}
The corresponding triple
\begin{equation}\label{WHS-triple}
(M,\sigma_m,H)
\end{equation}
is called a weighted Hamiltonian system.  

According to the weighted Hamiltonian system defined above, the equation of motion under the action-angle variable can be obtained

\begin{equation}\label{yundongfangcheng}
\begin{aligned}
  \dot{I} &= 0, \\
  \dot{\theta} &= m(\theta)^{-1/N}\,\omega(I),
\end{aligned}
\end{equation}
in which $\omega(I)=\nabla_{I}H(I)$.
\end{definition}

\begin{remark}
 If $m\equiv 1$, the form $\sigma_m$ reduces to the canonical symplectic form and the system coincides with the classical integrable Hamiltonian system.
\end{remark}

The phase space vector field of system \eqref{WHS-triple} is denoted as 
$$\XX (I,\theta)=(0,a(\theta)\omega(I)).$$
It follows that, with respect to the standard Lebesgue volume $\mathrm{d}I\,\mathrm{d}\theta$, 
the divergence of the vector field $\XX $ is  
\begin{equation*}
\nabla \cdot \XX (I,\theta)
= \sum_{k=1}^{N}\frac{\partial\big(a(\theta)\,\omega_k(I)\big)}{\partial\theta_k}
=\underbrace{\sum_{j=1}^{M}\frac{\partial (0)}{\partial I_j}}_{=0}
+ \sum_{k=1}^{N}\frac{\partial\big(a(\theta)\omega_k(I)\big)}{\partial\theta_k}
= \omega(I)\cdot \nabla_{\theta} a(\theta).
\end{equation*}
Hence, unless $m(\theta)$ is constant (equivalently, $a$ is constant), 
one has in general $\nabla \cdot \XX \neq 0$. 
Consequently, the Lebesgue volume is not preserved.
\begin{lemma}
There exists a family of weighted invariant measures
\begin{equation*}
\quad \mathrm{d} \mu_\ast := \rho(\theta)\,\mathrm{d}\theta\,\mathrm{d} I,\qquad 
\rho(\theta)=C\,m(\theta)^{1/N},\ C>0,\quad
\end{equation*}
such that for all $t$, if~$\Phi_t$~denotes the flow of system~\eqref{yundongfangcheng}, then
\begin{equation*}
\int_{\Omega\times\mathbb{T}^{N}} \varphi\circ \Phi_t \ \mathrm{d}\mu_\ast
  \;=\;
  \int_{\Omega\times\mathbb{T}^{N}} \varphi \ \mathrm{d}\mu_\ast,
  \qquad \forall\,\varphi\in C_c^\infty(\Omega\times\mathbb{T}^{N}).
\end{equation*}
\end{lemma}

\begin{proof}
Consider a trial density $\rho(\theta)$, leading to the measure
$\mathrm{d}\mu=\rho(\theta)\,\mathrm{d}\theta\,\mathrm{d} I$.  
Invariance is equivalent to the stationary continuity equation
\[
\nabla\cdot\big(\rho \XX \big)=0.
\]
Since $\XX $ has no $I$-component and $\rho$ is independent of $I$, one computes
\begin{equation*}
\nabla \cdot(\rho \XX )
=\sum_{k=1}^{N}\frac{\partial\big(\rho(\theta)\,a(\theta)\,\omega_k(I)\big)}{\partial\theta_k}
= \omega(I)\cdot\nabla_\theta\big(\rho(\theta)\,a(\theta)\big).
\end{equation*}
This expression vanishes identically provided that
\begin{equation*}
\rho(\theta)\,a(\theta)\equiv C\quad(\text{constant}).
\end{equation*}
Recalling $a(\theta)=m(\theta)^{-1/N}$, we obtain
\begin{equation*}
\rho(\theta)=C\,a(\theta)^{-1}=C\,m(\theta)^{1/N}.
\end{equation*}
Hence $\mathrm{d}\mu_\ast=\rho\,\mathrm{d}\theta\,\mathrm{d} I$ is an invariant measure.  
For all test functions $\varphi$, conservation along the flow follows by integration by parts or, equivalently, by verifying that the Jacobian of $\Phi_t$ with respect to $\mathrm{d}\mu_\ast$ is unity.
\end{proof}
For simplicity of notation in the following contents, let 
\begin{equation*}
  \bar{a}:=\frac{(2\pi)^N}{\int_{\TT^N} \rho(\theta)\,\mathrm{d}\theta}
=\frac{(2\pi)^N}{\!\int_{\TT^N} a(\theta)^{-1}\,\mathrm{d}\theta},
\qquad
b(\theta):=\frac{\bar {a}}{a(\theta)}-1.
\end{equation*}

Given an initial density
$f_0\in L^1(\Omega\times\TT^N,\dd\mu_\ast)$, we define the time $t$ pushforward (evolved) measure
\begin{equation}\label{def:Pt}
\dd P_t := f_t(I,\theta)\,\dd\mu_\ast,\qquad
f_t := f_0\circ \Phi_{-t}.
\end{equation}
Equivalently, for any measurable $B\subset\Omega\times\TT^N$,
\begin{equation*}
P_t(B)=\int_{\Omega\times\TT^N}\mathbf 1_{B}\big(\Phi_t(I,\theta)\big)\,
f_0(I,\theta)\,\dd\mu_\ast.
\end{equation*}
in which $\mathbf 1_{B}$ is an indicator function of set $B$.

Let
\begin{equation*}
W(I):=\int_{\TT^N} f_0(I,\theta)\,\rho(\theta)\,\dd\theta,\qquad
Z_\rho:=\int_{\TT^N}\rho(\vartheta)\,\dd\vartheta.
\end{equation*}
If $\int_{\Omega\times\TT^N} f_0\,\dd\mu_\ast=1$ (probability normalization),
then $\int_\Omega W(I)\,\dd I=1$. The weighted equilibrium measure is
\begin{equation}\label{def:Peq}
\dd P_{\mathrm{eq}}^{(m)}(I,\theta)
:= W(I)\,\dd I \;\otimes\; \frac{\rho(\theta)}{Z_\rho}\,\dd\theta .
\end{equation}
Consequently, for any bounded measurable $G$,
\begin{equation}\label{1.44444}
\int G \dd P_{\mathrm{eq}}^{(m)}
= \int_{\Omega}\left(\frac{\int_{\TT^N} G(I,\theta)\rho(\theta)\dd\theta}{\int_{\TT^N}\rho(\vartheta)\dd\vartheta}\right) W(I) \dd I.
\end{equation}

In this paper, we assume that the frequency mapping of system \eqref{WHS-triple} satisfies the following condition.
\begin{assumption}\label{ass:nonresonance}
Let $m\in C^{q}(\TT^{N})$ with $q>2N$, and let $H\in C^{2}(\Omega)$ on  $\Omega$.  
Write $\omega(I)=\nabla_{I}H(I)$, set $a(\theta)=m(\theta)^{-1/N}$ and $\rho(\theta)=m(\theta)^{1/N}$.  
Assume:
\begin{itemize}
  \item[\textnormal{(D1)}] There exist $\alpha>0$ and $2(N-1)<2\tau\leq q-N-1$ such that
  \begin{equation}\label{eq:Dioph}
  |n\cdot\omega(I)| \ \ge\ \alpha \, |n|^{-\tau}, \qquad \forall\, I\in\Omega,\ \forall\, n\in\mathbb{Z}^{N}\setminus\{0\}.
  \end{equation}
  \item[\textnormal{(D2)}] There exists $\lambda>0$ such that the minimal singular value satisfies
  \begin{equation*}
    \sigma_{\min}\big(D\omega(I)\big)\ \ge\ \lambda, \qquad \forall\, I\in\Omega,
  \end{equation*}
  and $D\omega, D^{2}\omega$ are bounded on $\Omega$.
\end{itemize}
\end{assumption}

\section{Angular Conjugate - Linearization of Weighted Flow}

In the weighted Hamiltonian system \eqref{WHS-triple}, the equation of motion for the angular variables is given by .
\begin{equation*}
  \dot\theta = a(\theta)\,\omega(I).
\end{equation*}
It is obvious that the rotation speed depends on the angle $\theta$ and the flow is not uniform.
However, when studying long-term statistical behaviors (such as time averages and weak convergence), we aim to reduce the system to a constant-speed linear flow, since Fourier analysis can be conveniently applied only in the case of constant-speed flows.

To achieve this reduction, we attempt to construct a conjugate transformation that depends on variable $I$
\begin{equation*}
\phi = \Psi_I(\theta) := \theta + \omega(I)\,v_I(\theta),
\end{equation*}
such that the transformed dynamics satisfies
\begin{equation*}
  \dot\phi = \bar a\,\omega(I),
\end{equation*}
with constant speed $\bar {a}\,\omega(I)$. Substituting this ansatz leads precisely to the cohomological equation
\begin{equation}\label{eq:cheq}
\omega(I)\cdot\nabla_\theta v_I(\theta)
= b(\theta),\qquad \int_{\TT^N} v_I(\theta)\,\dd\theta=0.
\end{equation}

\begin{lemma}\label{cunzaiweiyi}
Let $m\in C^{q}(\TT^{N})$ with $m>0$, where $q\ge \tau+N+1$ is an integer.
For the frequency map $\omega(I)=\nabla_I H(I)$, we assume that the nonresonance condition in Assumption~\ref{ass:nonresonance} is in force.
For each $I\in\Omega$, equation \eqref{eq:cheq} admits a unique solution $v_I$ in the zero-mean class, which has the Fourier expansion
\[
v_I(\theta)=\sum_{n\in\ZZ^N}\hat{v}_{I,n}\,e^{\ii n\cdot \theta}.
\]
If $b$ has Fourier expansion $b(\theta)=\sum_{n\in\ZZ^N} \hat{b}_n\,e^{\ii n\cdot\theta}$ with $\hat{b}_0=0$,
then the Fourier coefficients of $v_I$ satisfy
\begin{equation}\label{eq:Fourier-coeff-v}
\hat{v}_{I,n}=\frac{\hat{b}_n}{\ii\,n\cdot\omega(I)}\quad (n\neq 0),\qquad \hat{v}_{I,0}=0.
\end{equation}
\end{lemma}

\begin{proof}
  Assume $v_I(\theta)=\sum_{n\in\ZZ^N} \hat{v}_{I,n} e^{\ii n\cdot\theta}$, substituting into
the cohomological equation \eqref{eq:cheq} gives
\begin{equation*}
\sum_{n\in\ZZ^N} (\ii\,n\cdot\omega(I)\,)\hat{v}_{I,n}e^{\ii n\cdot\theta}
=\sum_{n\in\ZZ^N} b_n\,e^{\ii n\cdot\theta}.
\end{equation*}
Comparing coefficients (for $n=0$ we get $0=0$, consistent with the zero-mean condition), for $n\neq 0$ we obtain
\begin{equation*}
(\ii\,n\cdot\omega(I)\,)\hat{v}_{I,n}=b_n\quad\Longrightarrow\quad
\hat{v}_{I,n} = \frac{b_n}{\ii\,n\cdot\omega(I)}.
\end{equation*}

If another zero-mean solution $\tilde v_I$ existed, their difference $w_I:=v_I-\tilde v_I$ would satisfy
$\omega(I)\cdot\nabla_\theta w_I=0$ and $\int_{\TT^N} w_I\,\dd\theta=0$. Expanding in Fourier series yields
$(i\,n\cdot\omega(I))\,(w_I)_n=0$. By \eqref{eq:Dioph}, $n\cdot\omega(I)\ne 0$ for all $n\ne 0$, so
$(w_I)_n=0$; and the zero mean gives $(w_I)_0=0$, hence $w_I\equiv 0$. Uniqueness follows.

\end{proof}

However, the solution in Lemma \ref{cunzaiweiyi} is in fact only a formal one. Below we establish its regularity and verify that it is indeed a solution in the classical sense.

\begin{lemma}\label{sol:Regularity}
  The solution in Lemma \ref{cunzaiweiyi}, $v_I(\theta)=\sum_{n\in\ZZ^N} \hat{v}_{I,n} e^{\ii n\cdot\theta}$,
  \begin{equation*}
    \hat{v}_{I,n}=\frac{\hat{b}_n}{\ii\,n\cdot\omega(I)}\quad (n\neq 0),\qquad \hat{v}_{I,0}=0,
  \end{equation*}
  satisfies $\hat{v}_{I,n}\in C^{r}$, $r\in \ZZ$ and $0\leq r < q-\tau-N$.
\end{lemma}

\begin{proof}
If $m\in C^{q}$, then $a(\theta)=m(\theta)^{-1/N}\in C^{q}$, so $b\in C^{q}$ and its Fourier coefficients satisfy
$$|b_n|\le C_b(1+|n|)^{-q}.$$
By \eqref{eq:Dioph} and \eqref{eq:Fourier-coeff-v}, we have
\begin{equation*}
|\hat{v}_{I,n}|\ \le\ C\, (1+|n|)^{\tau-q},\qquad C=C(\alpha,C_b).
\end{equation*}
Hence, for $s\in \mathbb{Z}$, $0<s<q-\tau-\frac{N}{2}$,
\begin{equation*}
\sum_{n}(1+|n|)^{2s}\,|\hat{v}_{I,n}|^2\ \leq C^2 \sum_{n}(1+|n|)^{2s+2(\tau-q)}<\infty,
\end{equation*}
which means $v_I\in H^s(\TT^N)$ and $\|v_I\|_{H^s}$ is uniformly bounded in $I$. Furthermore, 
for any integer
$$0\ \le\ r\ <\ \ell-\tau-N,$$
the Sobolev embedding $H^{r+N/2+\varepsilon}\hookrightarrow C^r$ (for any small $\varepsilon>0$) yields
\begin{equation}\label{eq:vk-regularity}
v_I\in C^{r}(\TT^N)\quad\text{and}\quad \|v_I\|_{C^{r}}\ \le\ C_{r}\ \ \text{(uniformly in $I$)},
\end{equation}
which implies Fourier series with derivatives converge uniformly, so the termwise differentiation below is justified
\begin{equation*}
\nabla_\theta v_I(\theta)=\sum_{n\ne 0} (\ii\,n)\,(v_I)_n\,e^{\ii n\cdot\theta},
\end{equation*}
with uniform convergence. Then
\begin{equation*}
\omega(I) \cdot \nabla_{\theta} v_{I}(\theta)
=\sum_{n\neq 0} (\ii\,n\cdot\omega(I))\,(v_I)_n\,e^{\ii n\cdot\theta}
=\sum_{n\neq 0} b_n\,e^{\ii n\cdot\theta}=b(\theta),
\end{equation*}
showing that the solution of \eqref{eq:cheq} holds classically.
\end{proof}
Next, we establish the boundedness of $\partial_I\Psi_I(\theta)$ and provide a proof; this lemma will be invoked in the proof of Theorem~\ref{theorem3.1} in Section~\ref{sec:3}.

\begin{lemma}\label{boundofpsi}
  There exists a constant \(C>0\), independent of \(I\), such that
\begin{equation*}
  \sup_{(I,\theta)\in\Omega\times\TT^{N}}\|\partial_I\Psi_I(\theta)\|\ \leq\ C .
\end{equation*}
\end{lemma}

\begin{proof}
  As noted above,
\begin{equation*}
|\hat v_{I,n}|\ \le\ \alpha^{-1}|\hat b_n|\,|n|^{\tau}
\ \le\ C_b\,\alpha^{-1}\,(1+|n|)^{-(q-\tau)},
\end{equation*}
and therefore
\begin{equation*}
\|v_I\|_{L^\infty(\TT^N)}
\ \le\ \sum_{n\neq 0}|\hat v_{I,n}|
\ \le\ C\sum_{n\neq 0}(1+|n|)^{-(q-\tau)}
\ <\ \infty.
\end{equation*}
Differentiate the Fourier coefficients with respect to $I$:
\begin{equation*}
\partial_I\hat v_{I,n}
=\partial_I\!\left(\frac{\hat b_n}{\ii\,n\cdot\omega(I)}\right)
=-\,\hat b_n\,\frac{D\omega(I)^{\!\top}n}{(n\cdot\omega(I))^{2}}.
\end{equation*}
Using $\|D\omega\|_{L^\infty(\Omega)}<\infty$,
$|(n\cdot\omega(I))^{-2}|\le \alpha^{-2}|n|^{2\tau}$, and $|n|\le(1+|n|)$, we obtain
\begin{equation*}
\|\partial_I\hat v_{I,n}\|
\ \le\ C\,|\hat b_n|\,\frac{|n|}{|n\cdot\omega(I)|^{2}}
\ \le\ C\,|\hat b_n|\,|n|^{2\tau+1}
\ \le\ C\,(1+|n|)^{-(q-(2\tau+1))}.
\end{equation*}
Hence
\begin{equation*}
\|\partial_I v_I\|_{L^\infty(\TT^N)}
\ \le\ \sum_{n\neq 0}\|\partial_I\hat v_{I,n}\|
\ \le\ C\sum_{n\neq 0}(1+|n|)^{-(q-(2\tau+1))}.
\end{equation*}
Moreover, since
\begin{equation*}
\partial_I\Psi_I(\theta)
= D\omega(I)\,v_I(\theta)\ +\ \omega(I)\otimes\partial_I v_I(\theta),
\end{equation*}
we have
\begin{equation*}
\sup_{(I,\theta)}\|\partial_I\Psi_I(\theta)\|
\ \le\ \|D\omega\|_{L^\infty(\Omega)}\,\sup_{I}\|v_I\|_{L^\infty}
\ +\ \|\omega\|_{L^\infty(\Omega)}\,\sup_{I}\|\partial_I v_I\|_{L^\infty}
\ \le\ C,
\end{equation*}
where $C$ depends only on
$\alpha,\tau,q,C_b,\|D\omega\|_{L^\infty},\|\omega\|_{L^\infty}$, and $N$, and is independent of $I$. This completes the proof.
\end{proof}

We now state the main result of this subsection.
\begin{theorem}\label{theorem2.1}
The conjugacy $\phi = \Psi_I(\theta) := \theta + \omega(I)\,\nu_{I}(\theta)$ is a $C^{r}$ diffeomorphism such that the new angular variables form a constant-speed linear flow with frequency $\bar{a}\omega(I)$, and $v_{I}=\bar{a}^{-1}d\phi$ is a constant multiple of the Lebesgue measure on $\TT^{N}$.
\end{theorem}
\begin{proof}
  The derivative and Jacobian determinant of $\Psi_{I}$ are
  \begin{equation}\label{eq:Jacobian}
D_\theta\Psi_I(\theta)=I_{N}+\omega(I)\otimes\nabla_\theta v_I(\theta),
\qquad
\det D_\theta\Psi_I(\theta)=1+\omega(I)\cdot\nabla_\theta v_I(\theta)=\frac{\bar a}{a(\theta)}\,>\,0.
\end{equation}
Thus $\Psi_I$ is a local $C^{r}$ diffeomorphism and orientation-preserving. Furthermore, since
\begin{equation}\label{eq:degree-1}
\int_{\TT^N}\det D_\theta\Psi_I(\theta)\,\dd\theta
= \int_{\TT^N}\frac{\bar a}{a(\theta)}\,\dd\theta
= \bar a \int_{\TT^N} a(\theta)^{-1}\,\dd\theta
= (2\pi)^N,
\end{equation}
and $\TT^{N}$ is compact and connected, hence the local diffeomorphism $\Psi_{I}$ is a covering map, which implies that its topological degree is 
which means that the topological degree 
\begin{equation*}
  \frac{1}{(2\pi)^{N}}\int \det D_{\theta}\Psi_{I}=1.
\end{equation*}
As any covering map with topological degree 1 must be a global homeomorphism, it follows that $\Psi_{I}$ is a $C^{r}$ diffeomorphism.
Moreover, the following equation can be easily taken by \eqref{eq:Jacobian}, 
\begin{equation*}
\begin{aligned}
\dot\phi
&= D_\theta\Psi_I(\theta)\,\dot\theta
= \big(I_{N}+\omega(I)\otimes\nabla_\theta v_I(\theta)\big)\,a(\theta)\,\omega(I)\nonumber\\
&= a(\theta)\,\big(1+\omega(I)\cdot\nabla_\theta v_I(\theta)\big)\,\omega(I)
= \bar a\,\omega(I), 
\end{aligned}
\end{equation*}
and this means the new angle variable follows a constant-speed linear flow: $\phi(t)=\phi_0+\bar a\,\omega(I)\,t$.
On the other hand, for any test function
$h\in C(\TT^N)$, denoting $\nu_I = (\Psi_I)_{\#}\big(\rho(\theta)\,\dd\theta\big)$ as the pushforward measure, we have
\begin{equation*}
  \begin{aligned}
\int_{\TT^N} h(\phi)\,\dd\nu_I(\phi)
&= \int_{\TT^N} h\big(\Psi_I(\theta)\big)\,\rho(\theta)\,\dd\theta
= \int_{\TT^N} h(\phi)\,\frac{\rho\big(\Psi_I^{-1}(\phi)\big)}{\det D_\theta\Psi_I\big(\Psi_I^{-1}(\phi)\big)}\,\dd\phi \nonumber\\
&= \int_{\TT^N} h(\phi)\,\underbrace{\frac{\rho(\theta)}{\det D_\theta\Psi_I(\theta)}}_{=\ \rho(\theta)\,a(\theta)\,/\,\bar a\ =\ \bar a^{-1}}\,\dd\phi
= \bar a^{-1}\int_{\TT^N} h(\phi)\,\dd\phi.
  \end{aligned}
\end{equation*}
Here we used $\rho(\theta)=m(\theta)^{1/N}$ and $a(\theta)=m(\theta)^{-1/N}$, which give
\[
\frac{\rho(\theta)}{\det D_\theta\Psi_I(\theta)}
=\frac{\rho(\theta)}{\bar a/a(\theta)}
=\frac{\rho(\theta)\,a(\theta)}{\bar a}=\frac{1}{\bar a}.
\]
That is, $\nu_I=\bar a^{-1}\,\dd\phi$ is a constant multiple of the Lebesgue measure on $\TT^N$.

\end{proof}

\section{The convergence of statistical ensembles}\label{sec:3}

In this subsection, we present the principal results. We establish the convergence rate of the ensemble and the convergence of the ensemble itself.

\begin{theorem}\label{theorem3.1}
Assume that system~\eqref{WHS-triple} satisfies Assumption~\ref{ass:nonresonance} together with the hypotheses of Lemma~\ref{cunzaiweiyi}. Using the conjugacy $\phi = \Psi_{I}(\theta)$, define the flow
\begin{equation*}
    \Phi_{t}^{\theta}(I,\theta) = \phi_{0} + \bar{a}\,\omega(I)\,t .
\end{equation*}
The invariant measure is given by $\mathrm{d}\mu_\ast = \rho(\theta)\,\mathrm{d}\theta\,\mathrm{d}I$. Let $f_0 \in W_{I}^{1,1}(\Omega\times \TT^{N})$ be the initial density, and set $W(I):=\int_{\TT^{N}} f_0(I,\theta)\,\rho(\theta)\,\mathrm{d}\theta$. If $G \in C^{r}(\bar{\Omega} \times \TT^{N})$ with $r > N+1$ and $W \in W^{1,1}(\Omega)$, then there exists a constant $C>0$-depending only on $(\alpha,\tau,\lambda)$, the boundedness of $m$, and bounds on the derivatives of $\omega$-such that, for all $t \geq 1$,
\begin{equation}\label{eq:theend}
\Big|\int G\,\mathrm{d}P_t - \int G\,\mathrm{d}P_{\mathrm{eq}}^{(m)}\Big|
\ \le\ \frac{C}{t}\,\|G\|_{C^{r}}\,\big(\|W\|_{W^{1,1}}+\|W\|_{L^{1}}\big).
\end{equation}
\end{theorem}

\begin{proof}
For each fixed $I \in \Omega$, let
\begin{equation*}
\nu_I := (\Psi_I)_{\#}\big( f_0(I,\theta)\,\rho(\theta)\,\mathrm d\theta \big)
\end{equation*}
denote the measure on $\TT^N$ obtained by pushing forward the measure $f_0(I,\theta)\,\rho(\theta)\,\mathrm d\theta$ via the change of variables $\phi = \Psi_I(\theta)$. For any test function $h \in C(\TT^N)$,
\begin{equation*}
\int_{\TT^N} h(\phi)\,\mathrm d\nu_I(\phi)
\;=\; \int_{\TT^N} h\big(\Psi_I(\theta)\big)\, f_0(I,\theta)\,\rho(\theta)\,\mathrm d\theta .
\end{equation*}
Taking $\mathrm d\phi$ (Lebesgue measure) as the reference measure, $\nu_I$ admits the density
\begin{equation}\label{eq:tildef0-def}
\tilde f_0(I,\phi)
:= f_0\big(I,\Psi_I^{-1}(\phi)\big)\,\frac{\rho\big(\Psi_I^{-1}(\phi)\big)}{\det D_\theta\Psi_I\big(\Psi_I^{-1}(\phi)\big)}.
\end{equation}
Consequently,
\begin{equation*}
\int_{\TT^N} h(\phi)\,\mathrm d\nu_I(\phi)
\;=\; \int_{\TT^N} h(\phi)\,\tilde f_0(I,\phi)\,\mathrm d\phi .
\end{equation*}
On the other hand, by Theorem~\ref{theorem2.1} and its proof we have $(\Psi_I)_{\#}(\rho\,\mathrm d\theta)=\bar a^{-1}\,\mathrm d\phi$ and, moreover, $\det D_\theta\Psi_I=\bar a/a$ and $\rho=1/a$. Hence \eqref{eq:tildef0-def} simplifies to
\begin{equation}\label{eq:tildef0-simplified}
\tilde f_0(I,\phi)
= f_0\big(I,\Psi_I^{-1}(\phi)\big)\,\frac{\rho}{\det D\Psi_I}\Big|_{\theta=\Psi_I^{-1}(\phi)}
= \frac{1}{\bar a}\, f_0\big(I,\Psi_I^{-1}(\phi)\big).
\end{equation}
Taking $h(\phi) = G\big(I,\phi+\bar a\,\omega(I)\,t\big)$, we have
\begin{align*}
\int G\,\mathrm{d}P_t 
&= \int_{\Omega}\int_{\TT^N} G\big(I,\Phi_t^\theta(I,\theta)\big)\, f_0(I,\theta)\,\rho(\theta)\,\mathrm{d}\theta\,\mathrm{d}I\\
&= \int_{\Omega}\int_{\TT^N} G\big(I,\phi_0+\bar a\,\omega(I)\,t\big)\, \tilde f_0(I,\phi_0)\,\mathrm{d}\phi_0\,\mathrm{d}I.
\end{align*}
Expand $G(I,\cdot)$ in Fourier series and exchange sum and integration:
\begin{equation*}
\int G\,\dd P_t
=\sum_{n\in\ZZ^N}\int_{\Omega} G_n(I)\,e^{\,\ii t\,\bar a\,(n\cdot\omega(I))}
\underbrace{\Big(\int_{\TT^N} e^{\,\ii n\cdot\phi_0}\,\tilde f_0(I,\phi_0)\,\dd\phi_0\Big)}_{=:~M_n(I)}\,\dd I.
\end{equation*}
Combining equation \eqref{1.44444} with conjugacy $\phi = \Psi_{I}(\theta)$, we have
\begin{equation*}
\int G \dd P_{\mathrm{eq}}^{(m)}
= \int_{\Omega}\left(\frac{\int_{\TT^N} G(I,\theta)\rho(\theta)\dd\theta}{\int_{\TT^N}\rho(\vartheta)\dd\vartheta}\right) W(I) \dd I =\int_{\Omega}\left(\frac{1}{(2\pi)^N}\int_{\TT^N} G(I,\phi)\,\dd\phi\right)W(I)\,\dd I=\int_{\Omega} G_0(I)\,M_0(I)\,\dd I.
\end{equation*}
Then 
\begin{equation}\label{eq:main-diff-CN}
\int G\,\dd P_t-\int G\,\dd P_{\mathrm{eq}}^{(m)}
=\sum_{n\neq 0}\ \mathcal I_n(t),
\qquad
\mathcal I_n(t):=\int_{\Omega} A_n(I)\,e^{\,\ii t\,\Phi_n(I)}\,\dd I,
\end{equation}
where
\begin{equation}\label{eq:An-Phi-CN}
A_n(I):=G_n(I)\,M_n(I),\qquad
\Phi_n(I):=\bar a\,\big(n\cdot\omega(I)\big).
\end{equation}
Since
\begin{equation*}
|M_n(I)| \le \int_{\TT^N} \tilde f_0(I,\phi_0)\,\dd\phi_0
= \int_{\TT^N} f_0(I,\theta)\,\rho(\theta)\,\dd\theta \;=:\; W(I),
\end{equation*}
it follows that $M_n\in L^1(\Omega)$, and
\begin{equation*}
\|A_n\|_{L^1(\Omega)} \le \|G_n\|_{L^\infty(\Omega)}\,\|W\|_{L^1(\Omega)}.
\end{equation*}
In addition, since $f_0\in W^{1,1}_I(\Omega\times\TT^N)$ and
$\partial_I\Psi_I(\theta)$ is uniformly bounded on $\Omega\times\TT^N$  (by lemma \ref{boundofpsi}), it follows that for each fixed $n\in\ZZ^N$,
\begin{equation*}
\partial_I M_n(I)
=\int_{\TT^N} e^{\ii n\cdot \Psi_I(\theta)}
\Big[\partial_I f_0(I,\theta)
+\ii\,(n\!\cdot\!\partial_I\Psi_I(\theta))\,f_0(I,\theta)\Big]\rho(\theta)\,\dd\theta,
\end{equation*}
and hence
\begin{equation}\label{eq:dIMn-L1-CN}
\|\partial_I M_n\|_{L^1(\Omega)}
\ \le\ \|W^{(1)}\|_{L^1(\Omega)}
\ +\ |n|\,C_\Psi\,\|W\|_{L^1(\Omega)},
\qquad
W^{(1)}(I):=\!\int_{\TT^N}\!|\partial_I f_0(I,\theta)|\,\rho(\theta)\,\dd\theta,
\end{equation}
where $C_\Psi:=\|\partial_I\Psi_I\|_{L^\infty(\Omega\times\TT^N)}$. Consequently,
\begin{equation}\label{eq:W11-An-CN}
  \begin{aligned}
  \|A_n\|_{W^{1,1}(\Omega)}&
\ \leq\ \|\nabla_I G_n\|_{L^\infty}\,\|M_n\|_{L^1}
\ +\ \|G_n\|_{L^\infty}\,\|\partial_I M_n\|_{L^1}\\
&\ \leq\ C\,(1+|n|)^{-r}\left(\|W\|_{L^1}+\|W^{(1)}\|_{L^1}+|n|\,\|W\|_{L^1}\right).
  \end{aligned}
\end{equation}
Using the density of $C_c^\infty(\Omega)$ in $W^{1,1}(\Omega)$ on a Lipschitz domain, choose
$A_n^{(j)}\in C_c^\infty(\Omega)$ such that
\begin{equation*}
\|A_n^{(j)}-A_n\|_{W^{1,1}(\Omega)}\ \xrightarrow[j\to\infty]{}\ 0 .
\end{equation*}
Set
\begin{equation*}
\mathcal I_n^{(j)}(t):=\int_\Omega A_n^{(j)}\,e^{\ii t\Phi_n}\,\dd I.
\end{equation*}
Then
\begin{equation*}
\mathcal I_n^{(j)}(t)=\frac{1}{\ii t}\int_\Omega e^{\ii t\Phi_n}\,\Big[(\nabla\cdot V_n)\,A_n^{(j)} + V_n\!\cdot\nabla A_n^{(j)}\Big]\dd I,
\end{equation*}
where
\begin{equation*}
V_n(I)=\frac{\nabla_I\Phi_n(I)}{\|\nabla_I\Phi_n(I)\|^2}
=\frac{D\omega(I)^{\top}n}{\|D\omega(I)^{\top}n\|^2}\,\bar a^{-1}.
\end{equation*}
It follows that
\begin{equation*}
|\mathcal I_n^{(j)}(t)|
\ \le\ \frac{1}{t}\Big(\|\nabla\!\cdot V_n\|_{L^\infty}\,\|A_n^{(j)}\|_{L^1}
+\|V_n\|_{L^\infty}\,\|\nabla A_n^{(j)}\|_{L^1}\Big).
\end{equation*}
By Assumption~\ref{ass:nonresonance},
\begin{equation*}
\|V_n\|_{L^\infty(\Omega)}\ \lesssim\ |n|^{-1},\qquad
\|\nabla V_n\|_{L^\infty(\Omega)}\ \lesssim\ |n|^{-1},
\end{equation*}
we obtain
\begin{equation*}
|\mathcal I_n^{(j)}(t)|
\ \le\ \frac{C}{t}\,|n|^{-1}\,\Big(\|A_n^{(j)}\|_{L^1}+\|\nabla A_n^{(j)}\|_{L^1}\Big).
\end{equation*}
Next, on the one hand,
\begin{equation*}
|\mathcal I_n(t)-\mathcal I_n^{(j)}(t)|
=\Big|\int_\Omega (A_n-A_n^{(j)})\,e^{\ii t\Phi_n}\,\dd I\Big|
\ \le\ \|A_n-A_n^{(j)}\|_{L^1(\Omega)}
\ \xrightarrow[j\to\infty]{}\ 0.
\end{equation*}
On the other hand, since $\|A_n^{(j)}\|_{W^{1,1}}\to \|A_n\|_{W^{1,1}}$, we have
\begin{equation*}
\limsup_{j\to\infty} |\mathcal I_n^{(j)}(t)|
\ \le\ \frac{C}{t}\,|n|^{-1}\,\|A_n\|_{W^{1,1}(\Omega)}.
\end{equation*}
Therefore,
\begin{equation*}
  \begin{aligned}
|\mathcal I_n(t)| &\ \leq\ \frac{C}{t}\,|n|^{-1}\,\Big(\|A_n\|_{W^{1,1}(\Omega)}+\|A_n\|_{L^1(\Omega)}\Big)\\
&\ \leq\ \frac{C}{t}\,|n|^{-1}\,(1+|n|)^{-r}\,
\Big(\|W^{(1)}\|_{L^1}+\|W\|_{L^1}+|n|\,\|W\|_{L^1}\Big).
  \end{aligned}
\end{equation*}
Therefore,
\begin{equation*}
\sum_{n\neq 0} |\mathcal I_n(t)|
\ \le\ \frac{C}{t}\,\|G\|_{C^r}\,\big(\|W^{(1)}\|_{L^1}+\|W\|_{L^1}\big)
\ \le\ \frac{C}{t}\,\\|G\|_{C^r}\,\big(\|W\|_{W^{1,1}}+\|W\|_{L^1}\big),
\end{equation*}
which proves the theorem.
\end{proof}
By the preceding theorem, the statistical ensemble associated with system~(1.1) converges; in particular,
\begin{theorem}
  Under the hypotheses of~Theorem~\ref{theorem3.1},
\begin{equation*}
\lim_{t\to\infty}\int G\,\mathrm{d}P_t= \int G\,\mathrm{d}P_{\mathrm{eq}}^{(m)}.
\end{equation*}
\end{theorem}
\begin{proof}
  In fact, the result follows by letting $t \to \infty$ in equation~\eqref{eq:theend}.
\end{proof}

\section{Conclusion}\label{sec:conclusion}
Within the framework of weighted Hamiltonian systems, this paper develops a linearization of the nonuniform angular flow and establishes quantitative convergence of statistical ensembles, thereby laying an analytic foundation for future limit theorems of ensembles under stochastic perturbations. 

In action--angle coordinates $(I,\theta)\in\Omega\times\TT^{N}$ we introduced the weighted symplectic form
\[
\sigma_m \;=\; m(\theta)^{1/N}\sum_{j=1}^{N}\dd\theta_j\wedge\dd I_j.
\]
To analyze long-time statistical behavior, we solved the cohomological equation
\[
\omega(I)\cdot\nabla_\theta v_I(\theta)=b(\theta),\qquad \avg{v_I}=0,
\]
and constructed the $I$-dependent conjugacy $\phi=\Psi_I(\theta)=\theta+\omega(I)\,v_I(\theta)$, which linearizes the angular dynamics to the constant-speed flow $\dot\phi=\bar a\,\omega(I)$. On this basis, the time evolution of the ensemble can be rewritten as oscillatory integrals over the action variable $I$. Introducing the vector field
\[
V_n(I)=\frac{\nabla_I\Phi_n(I)}{\|\nabla_I\Phi_n(I)\|^2}
=\frac{D\omega(I)^{\top}n}{\|D\omega(I)^{\top}n\|^2}\,\bar a^{-1},
\]
and carrying out a vector-field integration by parts on Lipschitz domains—using $W^{1,1}$-approximation and a careful control of boundary contributions—we obtain a $t^{-1}$ decay for each Fourier mode. Summing over modes under the summability condition $r>N+1$ yields the main quantitative estimate: for any $G\in C^{r}(\bar\Omega\times\TT^{N})$ with $r>N+1$ and any $f_0\in W^{1,1}_I(\Omega\times\TT^N)$,
\[
\Big|\int G\,\dd P_t-\int G\,\dd P_{\mathrm{eq}}^{(m)}\Big|
\ \le\ \frac{C}{t}\,\|G\|_{C^{r}}\,\big(\|W\|_{W^{1,1}}+\|W\|_{L^{1}}\big).
\]
In future work, we may consider related results for nearly integrable Hamiltonian systems, or extend the present framework to more general integrable symplectic manifolds (where global action–angle coordinates may not exist) and to settings with boundaries or singularities.

\section*{Acknowledgment}
    
    The Author Yong Li was supported by National Natural Science Foundation of China (12071175 , 12471183 and 12531009).

\section*{Declarations}
    {\bf Conflict of interest} The authors declare that they have no conflict of interest.




\end{document}